\documentclass[9pt,fleqn]{article}
\usepackage{amssymb}
\usepackage{amscd}
\usepackage{amsmath}

\usepackage{graphicx}
\usepackage{amsthm}
\usepackage{amsfonts}
\usepackage{latexsym}
\usepackage{ stmaryrd }
\usepackage{dsfont}
\usepackage{amsthm}
\usepackage{euscript}

\newtheorem{theorem}{Theorem}[section]

\newtheorem{lemma}[theorem] {Lemma}
\newtheorem{definition}[theorem]{Definition}

\newtheorem{corollary}[theorem]{Corollary}
 
\theoremstyle{remark}

\numberwithin{equation}{section}

\title{On the functional equations for polylogarithms in one variable.}
\author{Daniil Rudenko}

\begin{document}
\maketitle
\begin{abstract}
We develop a new approach to the study of the functional equations satisfied by classical polylogarithms, inspired by Goncharov's conjectures. We prove a sharpened version of Zagier's criterion for such an equation and explain, how our approach leads to a very simple description of the equations in one variable, satisfied by dilogarithm and trilogarithm. 

Our main result is the complete description of the functional equations for weight four polylogarithm in one variable. 

\end{abstract}
\parskip 9pt
{\bf Acknowledgments}.
I would like to thank S. Bloch,  A. Suslin, A. Beilinson and A. Goncharov for invaluable discussions and suggestions. My special gratitude goes to H. Gangl for showing me his computations for polylogarithms in weight four.

\section{Introduction.}

\begin{definition}
The classical polylogarithm is a function on the unit disc $|z|<1$ given by a convergent power series
$$
Li_n(z)=\sum_{k=1}^{\infty}\frac{z^k}{k^n}
$$
and continued analytically to a universal covering of $\mathbb{C}P^1 \setminus \{0,1,\infty\}$ by an integral representation
$$
Li_n(z):=\int_{0}^{z}Li_{n-1}(t) \frac{dt}{t}, \\\\\ Li_1(z)=-log(1-z).
$$
The integer number $n$ is called the {\it weight} of the polylogarithm.  
\end{definition}

Our primary goal will be to study functional equations, satisfied by this function. To simplify their analysis it is instructive to neglect the terms, which are products of polylogarithms of lower weight. Another way to obtain similar functional equation is to consider "corrected" version of polylogarithm, which satisfies "clean" functional equations, defined by Zagier in \cite{Z}.

\begin{definition}
Zagier polylogarithm is defined by the following rule:
$$\mathcal{L}i_1(z)=log|z|,$$
for odd $n>1$ 
$$\mathcal{L}i_n(z)=\Re \left ( \sum_{k=0}^{n-1}\beta_k log^k|z| \cdot Li_{n-k}(z) \right ),$$
for even positive $n$
$$\mathcal{L}i_n(z)=\Re \left ( \sum_{k=0}^{n-1}\beta_k log^k|z| \cdot Li_{n-k}(z) \right ),$$
where $\dfrac{2x}{e^{2x}-1}=\sum_{k=0}^{\infty}\beta_k x^k$.
\end{definition}

Zagier polylogarithm is known to be single valued and continuous  on $\mathbb{CP}_1$ and real analytic on $\mathbb{CP}_1 \setminus \{0,1,\infty \}.$ Next, we define an abelian group, generated by single-variable functional equations, satisfied by $\mathcal{L}i_n(z)$. For arbitrary set $S$ denote by $\mathbb{Q}[S]$ a vector space with a basis, indexed by elements of $S.$

\begin{definition}
The group of single-variable functional equations for $\mathcal{L}i_n(z)$ is denoted by $\mathcal{A}_n.$ It is a sub-vector space of  $\mathbb{Q}[\mathbb{C}(t)]$ and characterized by the following property:

For $n_i \in \mathbb{Q}$ and $\varphi_i \in \mathbb{C}(t)$ the sum $\sum_{i=1}^{k} n_i [\varphi_i(t)]$ lies in  $\mathcal{A}_n$ if and only if $\sum_{i=1}^{k} n_i \mathcal{L}i_n(\varphi_i(t))=const.$ 
We will denote the inclusion map from $\mathcal{A}_n$ to $\mathbb{Q}[\mathbb{C}(t)]$ by $I_n.$
\end{definition}

In this work we will  describe groups $\mathcal{A}_n$ for weight less than  or equal to four.
The first result is a complete criterion for an element of  $\mathbb{Q}[\mathbb{C}(t)]$ to lie in $\mathcal{A}_n$. This is a sharpened version of Zagier's proposition 1 in \cite{Z}. To state it we need to introduce several axillary definitions.

Denote by $A$ a $\mathbb{Q}-$vector space of zero degree divisors on $\mathbb{CP}^1.$ For a function $\varphi (t)$ denote by $(\varphi)$ it's divisor. We will use inhomogeneous notation, where $[\infty] = 0 $ in $A$ and $A$ is a free vector space with the basis on all points of $\mathbb{A}^1_{\mathbb{C}}.$ Next, for each $p \in \mathbb{C}$ denote by $\psi_p$ an element of  $A^*$ dual to $p \in A.$ On a function $\varphi(t)$ it equals the order of zero (pole) of $\varphi$ in $p.$

Next, for vector space $V$ denote by $Sym^i V$ it's $i-$th symmetric power and by  $\bigwedge^i V$ it's $i-$th wedge power. Let $\mathbb{S}_*^{n-1,1}V$ be the following functor: $\mathbb{S}_*^{n-1,1}V \cong \dfrac{Sym^{n-1}V \otimes V}{ Sym^n{V}}.$ For $n=1$ we put $\mathbb{S}_*^{0,1}V \cong V.$  We used this symbol, because $\mathbb{S}_*^{n-1,1}(V)$ is dual to $\mathbb{S}^{2 1^{n-2}}V^*$ for Schur functor $\mathbb{S}^{2 1^{n-2}}.$ There is a map $J_n\colon \mathbb{Q}[\mathbb{C}(t)] \longrightarrow \mathbb{S}_*^{n-1,1} A$ defined by the formula $J_n([\varphi])=(\varphi)^{n-1} \otimes (1-\varphi).$ For instance, 
$$J_2[2t-t^2]=([0]+[2]-2[\infty]) \wedge (2[1]-2[\infty])=2[0]\wedge [1]+2[2]\wedge [1],$$
$$J_4[t-2]=([2]-[\infty])^3 \otimes ([3]-[\infty])=[2][2][2]\otimes [3].$$
We will sometimes put a subscript $n$ in symbol $\mathbb{Q}[\mathbb{C}(t)]_n,$ just to keep track of the weight of polylogarithm we will apply to rational functions in $\mathbb{Q}[\mathbb{C}(t)].$

\begin{theorem} \label{first}
The following sequence is exact
$$
0 \longrightarrow \mathcal{A}_n \stackrel{I_n}{\longrightarrow}  \mathbb{Q}[\mathbb{C}(t)]_n \stackrel{J_n}{\longrightarrow}  \mathbb{S}_*^{n-1,1} A.
$$
\end{theorem}

The statement of this result can be rephrased as follows: $\sum_{i=1}^{k} n_i \mathcal{L}i_n(\varphi_i(t))=const$ if and only if  $\sum_{i=1}^{k} (\varphi_i(t))^{n-1} \otimes (1-\varphi_i(t))=0$ in $\mathbb{S}_*^{n-1,1} A.$  Using the duality between $\mathbb{S}_*^{n-1,1} $ and the corresponding Schur functor one can give yet another formulation of our result: $$\sum_{i=1}^{k} n_i \mathcal{L}i_n(\varphi_i(t))=const$$  if and only if for arbitrary $p_1,...,p_n \in \mathbb{C}$ 
$$\sum_{i=1}^k n_i  \psi_{p_1}(\varphi_i)\cdot... \cdot  \psi_{p_{n-2}}(\varphi_i)\cdot \left ( \psi_{p_{n-1}}(\varphi_i)  \psi_{p_{n}}(1-\varphi_i) - \psi_{p_{n-1}}(1-\varphi_i)  \psi_{p_{n}}(\varphi_i) \right ) =0.$$

Our next result completely describes groups $\mathcal{A}_i$ for $i \leq 3.$

\begin{theorem} \label{second}
The maps  $J_1$, $J_2$ and $J_3$ are surjective. Equivalently,
$$\dfrac{\mathbb{Q}[\mathbb{C}(t)]_1}{\mathcal{A}_1} \cong A, \dfrac{\mathbb{Q}[\mathbb{C}(t)]_2}{\mathcal{A}_2} \cong A \wedge A, \dfrac{\mathbb{Q}[\mathbb{C}(t)]_3}{\mathcal{A}_3} \cong \dfrac{Sym^2A \otimes A}{Sym^3 A}.$$
\end{theorem}
From the proof of this result we will derive  two corollaries. The first one is so-called Roger's identity, claiming that for arbitrary rational function $f(t)$ 
$$\mathcal{L}i_2(f(t))= \sum_{a,b \in \mathbb{C}} \psi_a(f) \psi_b(1-f) \mathcal{L}i_2 \left ( \dfrac{t-a}{b-a} \right )+const.$$
For linear fractional $f(t)$ this identity reduces to the five-term relation.

The second identity, which can be found in works of Wojtkowiak, claims that for arbitrary rational function $f(t)$  such that both functions $f(t)$ and $1-f(t)$ are finite and invertible at $t=\infty,$
$$\mathcal{L}i_3(f(t)) = -\dfrac{1}{2}\sum_{a,b,c \in \mathbb{C}} \psi_a(f) \psi_b(f)  \psi_c(1-f) \mathcal{L}i_3 \left ( \dfrac{(t-a)(c-b)}{(t-b)(c-a)} \right) + const.$$
For $f(t)=\dfrac{(t-a)^2}{(t-b)(t-c)}$ this reduces to the classical Kummer's equation and for $f(t)=\dfrac{(t-a)(t-b)}{(t-c)(t-d)}$ this gives  Goncharov's equation.

Our last and the most novel result describes the image of the map $J_4.$ To give its formulation we need to give several more definitions. For four points $a,b,c,d \in \mathbb{CP}^1$ let us denote their cross-ratio by $r(a,b,c,d)=\dfrac{(a-b)(c-d)}{(c-b)(a-d)}.$
Then, $B_2(\mathbb{C})$ is a vector space $\mathbb{Q}[\mathbb{C}]/\mathcal{R}_2(\mathbb{C}),$ where $\mathcal{R}_2(\mathbb{C}) $ is generated by symbols $0$, $1$, $\infty$ and linear combinations 
$\sum_{i=0}^4 (-1)^i r(x_0,..., \hat{x_i}, ..., x_4).$ Note that every element of  $\mathcal{R}_2(\mathbb{C})$ gives rise to a relation between values of   $\mathcal{L}i_2$. This follows from the so-called five-term relation, found by Abel: 

$$\sum_{i=0}^4 (-1)^i \mathcal{L}i_2(r(x_0,..., \hat{x_i}, ..., x_4))=0.$$
Element of $B_2(\mathbb{C}),$ coming from $[x]\in \mathbb{Q}[\mathbb{C}]$ will be denoted by $\{x\}_2$.

\begin{theorem} \label{third}
The following sequence is exact:

$$0 \longrightarrow \mathcal{A}_4 \stackrel{I_4}{\longrightarrow}  \mathbb{Q}[\mathbb{C}(t)]_4 \stackrel{J_4}{\longrightarrow}  \mathbb{S}_*^{3,1} A  \stackrel{K_4}{\longrightarrow} (A \wedge A ) \otimes B_2(\mathbb{C}) \longrightarrow 0,$$
where 

$K_4(x_1 x_2 x_3 \otimes y)= \sum_{\sigma \in S_3} x_{\sigma(1)} \wedge x_{\sigma(2)} \otimes \{r(x_{\sigma(1)},y,x_{\sigma(3)}, \infty) \}_2.$
\end{theorem}

By {\it quadratic rational} function we mean a fraction of two polynomials in variable $t$ of degree $2$ with complex coefficients.  We call  a quadratic rational function $q(t)$  {\it special} if $q(\infty) \in \{0,1,\infty \}$ or $q(t)$ has a double root. The next corollary from the proof of theorem \ref{third} gives general functional equations for classical polylogarithm of weight four.

\begin{corollary}
For arbitrary rational function $f(t) \in \mathbb{C}(t)$ there exists  an integer $k$, rational numbers $a_i$ and special quadratic rational functions $q_i(t)$ such that
$$\mathcal{L}i_4(f(t)) = \sum_{i=1}^{k} a_i \mathcal{L}i_4 (q_i(t))+ const.$$
\end{corollary}

\section{Examples and motivation.}

We start this section with the proof of theorem \ref{second} and its corollaries. These results are very simple and serve as an illustration of our method of dealing with functional equations.  Next, we explain how general conjectures of Goncharov about mixed Tate motives lead to theorems \ref{first} and \ref{third}.

\begin{proof}
First, we need to prove that $J_1$ is surjective. This is true, since for arbitrary $a \in \mathbb{C}$, obviously, 
$J_1(t-x)=[x] \in A.$ Since elements of this type form a basis of $A$, $J_1$ is surjective. We remember that in our inhomogenious notation $[\infty]=0$

Similarly, $J_2$ is surjective, since $J_2(\dfrac{t-a}{b-a}) = [a] \wedge [b] \in A \wedge A = \mathbb{S}^2_*(A).$

The last claim is surjectivity of $J_3.$  Consider the following filtration on $\mathbb{S}^{2,1}_*V:$
$$\left [ \mathbb{S}^{2,1}_*V\right ]_1 \subseteq \left [ \mathbb{S}^{2,1}_*V\right ]_2 \subseteq \left [ \mathbb{S}^{2,1}_*V\right ]_3,$$
where $\left [ \mathbb{S}^{2,1}_*V\right ]_i$ is generated by tensors $[x_1][ x_2] \otimes [x_3] \in \mathbb{S}^{2,1}_*V$ such that between $x_j$ there are at most $i$ different. Let $gr_i \mathbb{S}^{2,1}_*V$ be associated graded factors $\left [ \mathbb{S}^{2,1}_*V\right ]_i / \left [ \mathbb{S}^{2,1}_*V\right ]_{i-1}.$ 

First note that $\left [ \mathbb{S}^{2,1}_*V\right ]_1$ vanishes, since $[a]^2 \otimes [a] \in Sym^3(V).$ Next, $J_3$ is surjective on $gr_2 \mathbb{S}^{2,1}_*V,$ since $J_3 \left ( \dfrac{t-a}{b-a}\right)=[a]^2 \otimes [b]$ and $[a][b] \otimes [a] =-\frac{1}{2}[a]^2 \otimes [b].$ Finally, we need to show that $J_3$ is surjective on $gr_3 \mathbb{S}^{2,1}_*V.$
This is true, since  $J_3 \left ( \dfrac{(t-a)(c-b)}{(t-b)(c-a)} \right )= ([a]-[b])^2 \otimes ([c]-[b]) = -2 [a][b] \otimes [c].$
\end{proof}

Next, we are going to prove Rogers identity and  that of Wojtkowiak.
To show that $$\mathcal{L}i_2(f(t))= \sum_{a,b \in \mathbb{C}} \psi_a(f) \psi_b(1-f) \mathcal{L}i_2 \left ( \dfrac{t-a}{b-a} \right )+const,$$
we need to prove that 
$$[f(t)]- \sum_{a,b \in \mathbb{C}} \psi_a(f) \psi_b(1-f) \left [ \dfrac{t-a}{b-a} \right ]$$
lies in $\mathcal{A}_2$. From theorem \ref{first} it is enough to show that it's image under $J_2$ vanishes. 
But this is obvious:
$$J_2(f(t))= (f) \wedge (1-f) =\sum_{a,b \in \mathbb{C}} \psi_a(f) \psi_b(1-f)  a \wedge b =  \sum_{a,b \in \mathbb{C}} \psi_a(f) \psi_b(1-f) J_2\left ( \dfrac{t-a}{b-a} \right ).$$

The trilogarithmic identity of Wojtkowiak claims that for $f(t)\in \mathbb{C}(t)$  such that both functions $f(t)$ and $1-f(t)$ are invertible at $t=\infty,$
$$\mathcal{L}i_3(f(t)) =-\frac{1}{2} \sum_{a,b,c \in \mathbb{C}} \psi_a(f) \psi_b(f)  \psi_c(1-f) \mathcal{L}i_3 \left ( \dfrac{(t-a)(c-b)}{(t-b)(c-a)} \right) + const.$$ 
This is equivalent to the fact that
$$[f(t)] + \frac{1}{2}\sum_{a,b,c \in \mathbb{C}}  \psi_a(f) \psi_b(f)  \psi_c(1-f) \left [ \dfrac{(t-a)(c-
b)}{(t-b)(c-a)} \right] $$
lies in $\mathcal{A}_3.$ By theorem \ref{first} it is enough to check that $J_3$ vanishes on this element.
First note that since $f$ and $1-f$  are invertible in $\infty,$ 
$\sum_{a \in \mathbb{C}} \psi_a(f)=0, \sum_{b \in \mathbb{C}} \psi_b(f)=0$ and  $\sum_{c \in \mathbb{C}} \psi_c(1-f)=0.$ From this we can prove the claim:
$$J_3(f)=\sum_{a,b,c \in \mathbb{C}}  \psi_a(f) \psi_b(f)  \psi_c(1-f) [a][b]\otimes[c]=$$
$$=-\frac{1}{2}\sum_{a,b,c \in \mathbb{C}}  \psi_a(f) \psi_b(f)  \psi_c(1-f) ([a]-[b])^2 \otimes ([c]-[b])=$$
$$=-\frac{1}{2}\sum_{a,b,c \in \mathbb{C}}  \psi_a(f) \psi_b(f)  \psi_c(1-f) J_3 \left (\dfrac{(t-a)(c-
b)}{(t-b)(c-a)} \right ).$$

   Our next goal is to show, how theorems \ref{first} and \ref{third} follow from general conjectures of Goncharov, explained in \cite{G}.
These conjectures explain the structure of the category $\mathcal{M}_T(F)$ of mixed Tate motives over field $F$. Note that even the existence of such a category with desirable properties is proved only for some types of fields.  Category   $\mathcal{M}_T(F)$  should be Tannakian, generated by tensor powers of simple object $\mathbb{Q}(1).$ By general properties of Tannakian categories there should exist a graded Lie coalgebra $\mathcal{L}_{MT}(F)$ such that  $\mathcal{M}_T(F)$ is equivalent to the category of representations of $\mathcal{L}_{MT}(F)$.  For simplicity we will denote it just by $\mathcal{L}(F).$ Beilinson and Deligne showed in \cite{BD} that for arbitrary $a \in F$ there should exist an element in $\mathcal{L}(F)$, denoted by
 $Li_n^\mathcal{M}(a),$ called motivic polylogarithm.  Their main property is that for $F= \mathbb{C}$ their Hodge realization gives classical polylogarithms, so all linear relations between motivic polylogarithms give rise to relations for numbers $Li_n(a).$
 
The Lie coalgebra  $\mathcal{L}(F)$ is positively graded by integers $i>0,$ and one can show that $\mathcal{L}_1(F)$ should be isomorphic to $F^*$, generated by motivic logarithms $log^M(a).$ Let $\mathcal{I}(F)$ be a coideal of elements in $\mathcal{L}(F)$ of degree  greater than one. It gives rise to an exact sequence of Lie coalgebras
$$0 \longrightarrow  F^* \longrightarrow \mathcal{L}(F) \longrightarrow \mathcal{I}(F) \longrightarrow 0.$$
The Freeness Conjecture of Alexander Goncharov claims that $\mathcal{I}(F)$ is cofree Lie coalgebra. Furthermore, its first homology group in degree $d$ is generated by elements $Li_d^\mathcal{M}(a),$ for $a \in F.$ Following Goncharov, we will denote this group by $B_d(F).$

It is reasonable to assume that  $\mathcal{L}(\mathbb{C})$ and $\mathcal{L}(\mathbb{C}(t))$ are related by the following exact sequence:
$$0 \longrightarrow  \mathcal{L}(\mathbb{C}) \longrightarrow\mathcal{L}(\mathbb{C}(t)) \longrightarrow \mathcal{L}_A \longrightarrow 0 ,$$
where $\mathcal{L}_A$ is a cofree Lie coalgebra, generated by $A.$ It is an object, graded dual to the graded free Lie algebra, generated by $A^*.$ The exact sequence above is an instance of Van Kampen theorem: fundamental group of a Riemann sphere with all points removed is freely generated by infinitesimal loops around all points. This exact sequence gives rise to the corresponding exact sequence of coideals
$$0 \longrightarrow  \mathcal{I}(\mathbb{C}) \longrightarrow\mathcal{I}(\mathbb{C}(t)) \longrightarrow \mathcal{I}_A \longrightarrow 0.$$
Since all the three coideals are cofree Lie coalgebras, according to Goncharov's Conjecture, only their zero and first homology are nontrivial. Application  of Hochshild-Serre spectral sequence leads to the following exact sequence:
$$0 \longrightarrow  H_1(\mathcal{I}(\mathbb{C})) \longrightarrow H_1(\mathcal{I}(\mathbb{C}(t))) \longrightarrow H_1(\mathcal{I}_A) \longrightarrow H_1(\mathcal{I}(\mathbb{C}),H_1(\mathcal{I}_A)) \longrightarrow  0.$$
Another part of Freeness Conjecture claims that in degree $d$ $H_1(\mathcal{I}(\mathbb{C}))\cong B_d(\mathbb{C})$ and $H_1(\mathcal{I}(\mathbb{C}(t)))\cong B_d(\mathbb{C}(t)).$ Since $B_d(\mathbb{C}(t))$ is a group, generated by motivic polylogarithms of rational functions, it is natural to assume that 
$$B_d(\mathbb{C}(t)) / B_d(\mathbb{C}) \cong \mathbb{Q}[\mathbb{C}(t)]_d/\mathcal{A}_d.$$
A simple application of Koszul resolution leads to the computation of $H_1(\mathcal{I}_A) $ which is isomorphic to $\mathbb{S}^{d-1,1}_*A.$  So, we come to an exact sequence 
$$0 \longrightarrow  \mathcal{A}_d \longrightarrow \mathbb{Q}[\mathbb{C}(t)]_d \longrightarrow \mathbb{S}^{d-1,1}_*A \longrightarrow H_1(\mathcal{I}(\mathbb{C}),H_1(\mathcal{I}_A)) \longrightarrow  0.$$
The first three terms of this sequence lead to the theorem \ref{first}. 

Due to degree reasons, $H_1(\mathcal{I}(\mathbb{C}),H_1(\mathcal{I}_A)) $ vanishes in degree less than $4$. This leads to theorem \ref{second}. 

In degree four $H_1(\mathcal{I}(\mathbb{C}),H_1(\mathcal{I}_A)) \cong A \wedge A \otimes B_2(\mathbb{C}).$
This leads to theorem \ref{third}.

\section{Proof of theorem \ref{first}.}
We need to show that the following sequence is exact:
$$
0 \longrightarrow \mathcal{A}_n \stackrel{I_n}{\longrightarrow}  \mathbb{Q}[\mathbb{C}(t)]_d \stackrel{J_n}{\longrightarrow}  \mathbb{S}_*^{n-1,1} A.
$$
We will prove it by induction on $n.$  
For arbitrary $p \in \mathbb{C}$ we introduce formal analogs of partial derivatives $D_p.$ Define 
$D_p \colon  Sym^n A \longrightarrow Sym^{n-1} A$
by formula 
$$D_p([x_1][x_2]...[x_{n}])= \sum_{i=1}^{n} \psi_p(x_i) \cdot [x_1][x_2]...\hat{[x_i]}...[x_{n}].$$
For us will be important that $D_pX^n=n\psi_p(X) \cdot X^{n-1}.$

\begin{lemma} \label{l1}
The map 
$\oplus D_p\colon Sym^n A \longrightarrow \bigoplus_{p \in \mathbb{C} }Sym^{n-1} A$
is injective.
\end{lemma}
\begin{proof}
For arbitrary points $p_i \in \mathbb{C}$ the composition $D_{p_1}\circ D_{p_2} \circ ... \circ D_{p_{n}}$ is a functional on $Sym^n A.$ It is easy to see that such elements generate the dual vector space to $Sym^n A.$ So, if for some element $Y \in Sym^n A$ $D_p Y=0$ for arbitrary $p$, then all linear functionals on  $Sym^n A$ vanish on $Y$ and we conclude that $Y=0.$
\end{proof}

Next, we let the operators $D_p$ act on all members of Koszul resolution $Sym^{n-k} A  \otimes \bigwedge^k A$ as $D_p \otimes id.$ It can be shown that $D_p$ form a chain map from Koszul resolution of weight $n$ to Koszul resolution of weight $n-1.$ So, the map $D_p \colon \mathbb{S}_*^{n-1,1}A  \longrightarrow \mathbb{S}_*^{n-2,1}A$ is correctly defined.

\begin{corollary} \label{c1}
The map 
$\oplus D_p: \mathbb{S}_*^{n-1,1}A \longrightarrow \bigoplus_{p \in \mathbb{C} } \mathbb{S}_*^{n-2,1}A$ is injective.
\end{corollary}
\begin{proof}
 $X_p$ form a chain map of Koszul resolutions so the following diagram is commutative.
$$
\begin{CD}
\mathbb{S}_*^{n-1,1}A@>\oplus D_p>>\bigoplus_{p \in \mathbb{C}}  \mathbb{S}_*^{n-2,1}A\\
      @VV K_{n-1}V   @VV \oplus K_{n-2} V     \\
Sym^{n-2}A \otimes A \wedge A@>\oplus D_p>> \bigoplus_{p \in \mathbb{C}} Sym^{n-3}A \otimes A \wedge A\\
\end{CD} $$
Since both vertical and the lower horizontal maps are injective, the upper horizontal map is injective as well.
\end{proof}

Now we introduce similarly denoted maps $D_p$ from  $\mathbb{Q}[\mathbb{C}(t)]_n$ to $\mathbb{Q}[\mathbb{C}(t)]_{n-1}.$ For $n\geq3$ we put
$$D_p[\varphi]_n=\psi_p(\varphi) \cdot [\varphi]_{n-1}$$ 
and for $n=2$ we put 
$$D_p[\varphi]_2=\psi_p(\varphi) \cdot [1-\varphi]_1-\psi_p(1-\varphi) \cdot [\varphi]_1.$$ 
It is easy to check that $D_p \circ J_n=J_{n-1} \circ D_p.$ Now theorem \ref{first} is a simple consequence of the following lemma:
\begin{lemma}\label{l2}
Element $F \in \mathbb{Q}[\mathbb{C}(t)]_n$ lies in $\mathcal{A}_n$ if and only $D_p F$ lies in $\mathcal{A}_{n-1}$ for every $p.$
\end{lemma}

\begin{proof}
The "only if" part is the direct corollary of theorem 1.17 in \cite{G1}. The proof is based on the analysis of singularities of the function $\dfrac{d\mathcal{L}i_n(F)}{dt}.$

The "if" part is slightly more tricky. For arbitrary abelian group $G$ denote by $G_{\mathbb{Q}}$ a $\mathbb{Q}-$vector space $G \otimes_{\mathbb{Z}} \mathbb{Q}.$ We first recall the well-known criteria of Zagier for polylogarithmic functional equations. 
It claims that if for some $n_i \in \mathbb{Z}$ and $\varphi_i \in \mathbb{C}(t)$ equality 
$$\sum n_i [\varphi_i]^{n-2} \otimes  [\varphi_i] \wedge [1-\varphi_i]=const$$ 
holds in $Sym^{n-2}\mathbb{C}(t)^*_{\mathbb{Q}} \otimes \mathbb{C}(t)^*_{\mathbb{Q}} \wedge \mathbb{C}(t)^*_{\mathbb{Q}},$
 then $\sum n_i \mathcal{L}i_n(\varphi_i(t))=const,$ so $\sum n_i [\varphi_i] \in \mathcal{A}_n.$ Here by $const$ in 
$Sym^{n-2}\mathbb{C}(t)^*_{\mathbb{Q}} \otimes \mathbb{C}(t)^*_{\mathbb{Q}} \wedge \mathbb{C}(t)^*_{\mathbb{Q}}$ we mean an element in $Sym^{n-2}\mathbb{C}^*_{\mathbb{Q}} \otimes \mathbb{C}^*_{\mathbb{Q}} \wedge \mathbb{C}^*_{\mathbb{Q}}$

Let us suppose that for some $X=\sum n_i [\varphi_i]$ for arbitrary $p \in \mathbb{C}$ holds $D_p(X)=0.$ We will deduce the fact that $X \in \mathcal{A}_n$ from Zagier criterion. Denote the element $\sum n_i [\varphi_i]^{n-2} \otimes  [\varphi_i] \wedge [1-\varphi_i]$ by $Z(X).$  

We will consider two sets of functionals $\alpha \colon \mathbb{C}(t)^*_{\mathbb{Q}}\longrightarrow \mathbb{Q},$ which we will call "constantial" and "valuational." To construct constantial functionals, consider some(uncountable) basis  of    $\mathbb{C}^*_{\mathbb{Q}}/\mathbb{Q}$ and take the dual basis. Then each element of this dual basis can be extended to whole $ \mathbb{C}(t)^*_{\mathbb{Q}}$ by letting it vanish on all monomials $(t-\lambda).$ In this way we get constantial functionals. Valuational functionals are just $\psi_p$ for $p \in \mathbb{C}.$  If each constantial and each valuational functional annulates an element of $ \mathbb{C}(t)^*_{\mathbb{Q}}$, then this element vanishes. Similarly, to show that $Z(X)=0$  it is enough to prove that all functionals $\psi_1\cdot ....\cdot \psi_{n-2} \otimes \psi_{n-1} \wedge \psi_n$ vanish on $Z(X),$ where each $\psi_i$ is either constantial  or valuational. We prove this statement by induction on the number $C$ of constantial functionals.  

Applying the "only if" part of the lemma to each $D_p X$ $n-1$ times, we deduce that if all $\psi_i$ are valuational, than any functional $\psi_1\cdot ....\cdot \psi_{n-2} \otimes \psi_{n-1} \wedge \psi_n$ vanishes on $Z(X).$ This proves the base of induction $C=0.$ Now, let us suppose that for $s=k-1$ the statement is proved. The value of the functional $\psi_1\cdot ....\cdot \psi_{n-2} \otimes \psi_{n-1} \wedge \psi_n$  on $Z(X)$ does not change after a permutation of the first $n-2$ functionals and change sign after the transposition of the last two. Moreover, the following sum vanishes on $Z(X):$
$$\psi_1\cdot ....\cdot \psi_{n-3}\cdot \psi_{n-2} \otimes \psi_{n-1} \wedge \psi_n+$$
$$\psi_1\cdot ....\cdot \psi_{n-3}\cdot \psi_{n} \otimes \psi_{n-2} \wedge \psi_{n-1}+$$
$$\psi_1\cdot ....\cdot \psi_{n-3}\cdot \psi_{n-1} \otimes \psi_{n} \wedge \psi_{n-2}.$$
From this it follows that it is enough to show that under the induction hypothesis 
functionals $\psi_1\cdot ....\cdot \psi_{n-2} \otimes \psi_{n-1} \wedge \psi_n$ vanishes on $Z(X)$ with $\psi_{n-1}$ constantial and $\psi_n$--- valuational.

The last claim follows from the following identity: for arbitrary constantial functional $\alpha \colon  \mathbb{C}^*_{\mathbb{Q}} \longrightarrow \mathbb{Q}$, $p \in \mathbb{C}$ and $\varphi(t) \in \mathbb{C}(t)$
$$\alpha(\varphi)\psi_p(1-\varphi)-\alpha(1-\varphi)\psi_p(\varphi)=$$ 
$$\sum_{q \in \mathbb{C}} \alpha(p-q)(\psi_p(\varphi) \psi_q(1-\varphi)-\psi_p(1-\varphi)\psi_q(\varphi)).$$
Indeed, from this identity we deduce that  $\psi_1\cdot ....\cdot \psi_{n-2} \otimes \alpha \wedge \psi_p$ equals to $\sum_{q \in \mathbb{C}} \alpha(p-q)\psi_1\cdot ....\cdot \psi_{n-2} \otimes \psi_p \wedge \psi_q$ on $Z(X)$. The latter number vanishes by the induction hypothesis. 

To complete the proof of the lemma it remains only to prove the identity above. For this let
$$\varphi(t)=A_{\varphi}\prod_{q \in \mathbb{C}}(t-q)^{\psi_q(\varphi)}, $$
$$1-\varphi(t)=A_{1-\varphi}\prod_{q \in \mathbb{C}}(t-q)^{\psi_q(1-\varphi)}.$$
Since $\alpha$ is constantial, it vanishes on all monomials $(t-\lambda),$ so $\alpha(\varphi)= \alpha(A_\varphi)$ and $\alpha(1-\varphi)= \alpha(A_{1-\varphi})$.  
The function 
$\dfrac{\varphi(t)^{\psi_p(1-\varphi)}}{(1-\varphi(t))^{\psi_p(\varphi)}}$ equals to $1$ in $p,$ 
so
$$\prod_{q \in \mathbb{C}}(p-q)^{\psi_p(\varphi) \psi_q(1-\varphi)-\psi_p(1-\varphi)\psi_q(\varphi)}=\dfrac{A_{\varphi}^{\psi_p(1-\varphi)}}{A_{1-\varphi}^{\psi_p(\varphi)}}.$$
The identity follows by taking the value of functional $\alpha$ of both parts.
\end{proof}

Now we derive theorem \ref{first} from lemma \ref{l2}. 

\begin{proof}
We prove it by induction on $n.$
For $n=1$ we need to show that the sequence 
$$
0 \longrightarrow \mathcal{A}_1 \stackrel{I_1}{\longrightarrow}  \mathbb{Q}[\mathbb{C}(t)]_1 \stackrel{J_1}{\longrightarrow}  A
$$
is exact. An element $F \in  \mathcal{A}_1$ is a linear combination $\sum n_i [\varphi_i],$ such that 
$\sum n_i log|\varphi_i(t)|=const.$ We may assume that $n_i$ are integers. From the previous identity it follows that $log \left | \prod \varphi_i^{n_i} \right |=const.$ By Liouville's theorem we deduce that  $\prod \varphi_i^{n_i}=const,$ so $\sum n_i (\varphi_i)=0$ in $A.$ This means that $J_1 \circ I_1=0.$ Similarly one can show that $Ker(J_1)=Im(I_1).$

Next suppose that theorem \ref{first} is proved for $n=k-1.$ Let us show that $J_k \circ I_k=0.$ For arbitrary $F \in  \mathbb{Q}[\mathbb{C}(t)]_k$ lying in the image of $\mathcal{A}_k$ and for any $p \in  \mathbb{C}$ by lemma \ref{l2}  $D_p F$ lies in $\mathcal{A}_{k-1},$
so  by inductive hypothesis $D_p \circ J_k F= J_{k-1} \circ D_p F=0.$ By corollary \ref{c1} $\oplus D_p$ is injective, so $J_k F=0.$ To check that $Ker(J_k)=Im(I_k)$ take $F \in Ker(J_k).$ Then $D_p F \in Ker(J_{k-1}),$ so $D_p F \in \mathcal{A}_{k-1}.$ By lemma \ref{l2} we deduce that $F \in \mathcal{A}_k.$
\end{proof}

\section{Proof of theorem \ref{third}}
We need to show that the  following sequence is exact:
$$0 \longrightarrow \mathcal{A}_4 \stackrel{I_4}{\longrightarrow}  \mathbb{Q}[\mathbb{C}(t)]_4 \stackrel{J_4}{\longrightarrow}  \mathbb{S}_*^{3,1} A  \stackrel{K_4}{\longrightarrow} (A \wedge A ) \otimes B_2(\mathbb{C}) \longrightarrow 0,$$
where 
$K_4(x_1 x_2 x_3 \otimes y)= \sum_{\sigma \in S_3} x_{\sigma(1)} \wedge x_{\sigma(2)} \otimes \{r(x_{\sigma(1)},y,x_{\sigma(3)}, \infty) \}_2.$

We will introduce  a filtration on $\mathbb{S}^{3,1}_*V$ similar to the one, we used in the proof of theorem \ref{second}:
$$\left [ \mathbb{S}^{3,1}_*V\right ]_1 \subseteq \left [ \mathbb{S}^{3,1}_*V\right ]_2 \subseteq \left [ \mathbb{S}^{3,1}_*V\right ]_3\subseteq \left [ \mathbb{S}^{3,1}_*V\right ]_4,$$
where $\left [ \mathbb{S}^{3,1}_*V\right ]_i$ is generated by tensors $[x_1][ x_2][ x_3] \otimes [x_4] \in \mathbb{S}^{3,1}_*V$ such that between $x_j$ there is at most $i$ different. Let $gr_i \mathbb{S}^{3,1}_*V$ be associated graded factors $\left [ \mathbb{S}^{3,1}_*V\right ]_i / \left [ \mathbb{S}^{3,1}_*V\right ]_{i-1}.$ 
For simplicity, we denote $gr_i \mathbb{S}^{3,1}_*A$ just by $gr_i$.  Let $\overline{gr_i}$ denote the factor of $gr_i$ by the image of morphism $J_4.$

By theorem \ref{first}, the sequence 
$$0 \longrightarrow \mathcal{A}_4 \stackrel{I_4}{\longrightarrow}  \mathbb{Q}[\mathbb{C}(t)]_4 \stackrel{J_4}{\longrightarrow}  \mathbb{S}_*^{3,1} A  $$
is exact. 
Next, $K_4$ is surjective, since $K_4\left (\frac{1}{2}[a]^2[b]\otimes \left [ \dfrac{a-bz}{1-z} \right ] \right)=a\wedge b \otimes \{ z \}_2.$
\begin{lemma}\label{l41}
The composition $K_4 \circ J_4$ vanishes.
\end{lemma}
\begin{proof}
To show that $K_4 \circ J_4=0$ let us recall the Rogers identity: for arbitrary $f(t) \in \mathbb{C}(t)$ the following equality holds in $B_2(\mathbb{C}):$
$$\{f(t)\}_2-\{f(\infty)\}_2=\sum_{a,b\in \mathbb{C}}\psi_a(f)\psi_b(1-f)\{r(b,t,a,\infty)\}_2.$$
From this equality it follows that
$$K_4 \circ J_4([f])=K_4([f]^3\otimes[1-f])=$$
$$\sum_{x_1,x_2,x_3,y\in\mathbb{C}}\psi_{x_1}(f)\psi_{x_2}(f)\psi_{x_3}(f)\psi_y(1-f)K_4(x_1x_2x_3\otimes y)=$$
$$\sum_{\sigma \in S_3}\sum_{x_1,x_2,x_3,y\in\mathbb{C}}\psi_{x_1}(f)\psi_{x_2}(f)\psi_{x_3}(f)\psi_y(1-f) x_{\sigma(1)} \wedge x_{\sigma(2)} \otimes \{r(x_{\sigma(1)},y,x_{\sigma(3)}, \infty) \}_2=$$
$$3\cdot \sum_{x_1,x_2,x_3,y\in\mathbb{C}}\psi_{x_1}(f)\psi_{x_2}(f)\psi_{x_3}(f)\psi_y(1-f) x_{1} \wedge x_{2} \otimes (\{r(x_1,y,x_3, \infty) \}_2-\{r(x_2,y,x_3, \infty) \}_2)=$$
$$3\cdot \sum_{x_1,x_2\in\mathbb{C}}\psi_{x_1}(f)\psi_{x_2}(f)  x_{1} \wedge x_{2} \otimes (\{f(x_1) \}_2-\{f(x_2) \}_2).$$
The last sum vanishes, since if $\psi_{x_1} \neq 0,$ then  $\{x_1\}_2 = 0$ and if $\psi_{x_2} \neq 0,$ then  $\{x_2\}_2 = 0.$
\end{proof}

It remains only to show that $Im J_4=Ker K_4$. 

\begin{lemma}\label{l42}
 $\overline{gr_1}=0,$  $\overline{gr_2}=0,$  $\overline{gr_4}=0.$
\end{lemma}
\begin{proof}
$\left [ \mathbb{S}^{3,1}_*V\right ]_1=0,$ since all elements $aaa\otimes a \in Sym^4 A,$ so $\overline{gr_1}=0.$
Next, $J_4\left [ \dfrac{t-a}{b-a}\right ]=[a]^3\otimes [b], J_4\left [ \dfrac{t-a}{t-b}\right ]=-[a-b]^3\otimes [b],$ so
$[a]^3\otimes [b]=0$ and $[a]^2[b] \otimes [b]=[a] [b]^2\otimes [b] \in \overline{gr_2}.$ Since in $\mathbb{S}_*^{3,1}A $ holds equality $ 3[a]^2 [b]\otimes[a]+[a]^3\otimes[b]=0,$ we deduce that $\overline{gr_2}=0.$ 

Let $a_1,a_2,a_3,b_1,b_2,b_3$ be six points on $\mathbb{CP}^1$ such that there exists a projective involution $I$ which translates $a_i$ to $b_i$.  Let's denote by
$\left| \begin{array}{ccc}
a_1 & a_2 & a_3 \\
b_1 & b_2 & b_3 \\
 \end{array} \right|$
 the rational function $\dfrac{(t-a_1)(t-b_1)(a_3-a_2)(a_3-b_2)}{(t-a_2)(t-b_2)(b_3-a_1)(b_3-b_1)}.$ This is exactly the unique quadratic function $q$ such that $q \circ I=q$ and $q(a_1)=0,q(a_2)=\infty, q(a_3)=1.$ 
 
It is easy to see that in $\overline{gr_4}$ the following holds:
$J_4\left [ \left| \begin{array}{ccc}
a_1 & a_2 & a_3 \\
b_1 & \infty & b_3 \\
 \end{array} \right| \right]=([a_1]+[b_1]-[a_2])^3\otimes([a_3]+[b_3]-[a_2])=-6[a_1][b_1][a_2]\otimes ([a_3]+[b_3]).$
 For four points $a,b,c,d \in \mathbb{CP}^1$ consider three involutions: $I_a$ which maps $a$ to $\infty$ and $b$ to $c$, $I_b$ which maps $b$ to $\infty$ and $a$ to $c$, and $I_c$ which maps $c$ to $\infty$ and $b$ to $a$. Then, obviously, these involutions are commuting and $I_a \circ I_b \circ I_c =id.$ From this it follows that involution $I_a$ maps $I_b(d)$ to $I_c(d)$. Consider the following element of $  \mathbb{S}^{3,1}_*A:$
 $$X=\left [ \left| \begin{array}{ccc}
a & b & d \\
c & \infty & I_b(d) \\
 \end{array} \right| \right]+
 \left [ \left| \begin{array}{ccc}
a & c & d \\
b & \infty & I_c(d) \\
 \end{array} \right| \right]-
 \left [ \left| \begin{array}{ccc}
b & a & I_b(d) \\
c & \infty & I_c(d) \\
 \end{array} \right| \right].$$ 
 Then $J_4(X)=-6[a][b][c]\otimes([d]+[I_b(d)])-6[a][b][c]\otimes([d]+[I_c(d)])+6[a][b][c]\otimes([I_c(d)]+[I_b(d)])=-12[a][b][c]\otimes [d].$ From this it follows that $\overline{gr_4}=0.$
\end{proof}

In view of lemma \ref{l42} we need only to show that $K_4:\overline{gr_3} \longrightarrow  (A \wedge A ) \otimes B_2(\mathbb{C})$ is an isomorphism. For this we will construct a map $L_4$ in reverse direction. We define it on $A \otimes A \otimes \mathbb{Z}[\mathbb{C}] $ by formula 
$$L_4([a]\otimes [b] \otimes \{z\}_2)=\dfrac{1}{2}[a]^2[b] \otimes \left [  \dfrac{a-bz}{1-z} \right ].$$
Obviously, $K_4 \circ L_4=id,$ so theorem \ref{third} will be proven if we show that $L_4$ is correctly defined on $(A \wedge A ) \otimes B_2(\mathbb{C})$. For this we need to check four identities in $\overline{gr_3}:$ three more simple and one much harder. The last one is the following "five-term" identity:
$$
L_4\left ([a]\wedge[b] \otimes\left(\sum_{i=0}^4 (-1)^i r(x_0,..., \hat{x_i}, ..., x_4)\right )\right)=0.
$$
The following lemma contains three more simple identities:
\begin{lemma}\label{l43}
In $\overline{gr_3}$ theimage of the following three elements under the map $L_4$ vanish:
$$
[a]\otimes [b] \otimes \{z\}_2+[b]\otimes [a] \otimes \{z\}_2,
$$
$$
[a]\otimes [b] \otimes \{z\}_2+[a]\otimes [b] \otimes \left \{\frac{1}{z} \right \}_2,
$$
$$
[a]\otimes [b] \otimes \{z\}_2+[a]\otimes [b] \otimes \left \{1- z \right \}_2.
$$
\end{lemma}
\begin{proof}
The space $\overline{gr_3}$ is generated by elements $[a]^2 [b] \otimes [c],$ because $2[a][b][c]\otimes [a]=-[a]^2 [b] \otimes [c]-[a]^2 [c] \otimes [b].$ Since $J_4\left [ \dfrac{(t-a)(c-b)}{(t-b)(c-a)} \right ]=3[a][b]^2\otimes [c]- 3 [a]^2[b]\otimes[c],$ in $\overline{gr_3} $ we have $[a][b]^2\otimes [c]= [a]^2[b]\otimes[c].$ It is easy to see that $J_4\left [ \dfrac{(t-a)(t-b)}{(d-a)(d-b)} \right ]=([a]+[b])^3\otimes([d]+[a+b-d])=6[a]^2[b] \otimes ([d]+[a+b-d]).$ From this the equality $L_4\left((a \otimes b + b\otimes a) \otimes \{z\}_2\right)=0$ follows. The second identity follows from the fact that $L_4(a \wedge b \otimes \{z\})=L_4(b \wedge a \otimes \{\frac{1}{z}\}).$
Finally,  $J_4\left [ \dfrac{(t-b)^2z(1-z)}{(t-a)(a-b)} \right ]=\left ( 2[b]-[a]\right )^3 \otimes \left ( \left [ \dfrac{a-bz}{1-z} \right ]+ \left [ \dfrac{a-b+bz}{z} \right ] - [a] \right )=$
$=-6[b][a]^2 \otimes \left ( \left [ \dfrac{a-bz}{1-z} \right ]+ \left [ \dfrac{a-b+bz}{z} \right ]\right ).$ This is equivalent to the third identity.
\end{proof}
 
It remains to show that the "five term" identity holds. For this we need to do some preparation. Consider arbitrary distinct five points $x_1,...,x_5$ on $\mathbb{P}_{\mathbb{C}}^1$. They determine fifteen other points on $\mathbb{P}_{\mathbb{C}}^1$ in the following way: for each choice of two distinct pairs out of these points  one can construct an image of the fifth point under the involution, interchanging points in pairs. For instance, if we pair $x_1$ with $x_2$ and $x_3$ with $x_4$ we construct a point $y$ such that there exist an involution sending $x_1$ to $x_2$, $x_3$ to $x_4$ and $x_5$ to $y$. This point will be denoted by symbol $[x_1 x_2| x_3x_4]$. Note that this symbol makes sense only when all the five initial points are given.  It is not changed after interchanging of the first pair of elements, second pair of elements and two pairs with each other.   The fifteen  points constructed all together form an interesting configuration on $\mathbb{P}_{\mathbb{C}}^1.$ The following lemma summarizes its main properties.
\begin{lemma} \label{l44}
 An involution, interchanging $x_1$ with $x_2$ and $x_3$ with $x_4$ also interchanges 
$[x_1 x_2| x_3x_4]$
 with
 $x_5.$
 The same involution interchanges 
$[x_1 x_3| x_2 x_4]$
 with 
$[x_1 x_4| x_2 x_3].$
 
 An involution, which fixes $x_1$ and interchanges $x_2$ with $x_3$, also interchanges   $[x_1 x_4| x_2 x_3]$
 with   
$[x_1 x_5| x_2 x_3].$
  \end{lemma}
\begin{proof}
The first statement of the lemma was checked in the proof of lemma \ref{l42} via the trick with three commuting involutions. Let us check the second statement. For this denote point $[x_1 x_4| x_2 x_3]$
 by $A$
 and 
$[x_1 x_5| x_2 x_3]$
 by $B.$ 
 Cross ratio is preserved by a projective involution, so $r(x_1,x_2,x_3,A)=r(x_4,x_3,x_2,x_5)=r(B,x_2,x_3,x_1)=r(x_1,x_3,x_2,B).$ From this the claim of the lemma follows.
 \end{proof}
Using this lemma we will construct three types of quadratic rational functions, which we will use to construct an element in $\mathbb{Q}[\mathbb{C}(t)]_4$ whose image under $J_4$ will coincide with an image under the map 
$L_4$ of $$[a]\wedge[b] \otimes\left(\sum_{i=0}^4 (-1)^i r(x_0,..., \hat{x_i}, ..., x_4)\right ).$$
Suppose that distinct five points $x_1,...,x_5 \in \mathbb{P}_\mathbb{C}^1$ are given. Let $q[x_1 x_2| x_3 x_4]$ be a quadratic function with zeros $x_1$ and $x_2$, poles $x_3$ and $x_4$ and equal to $1$ in $x_5$ and $[x_1 x_2| x_3x_4]$. 
Next, $q[x_1 x_2|| x_3 x_4]$ will be a quadratic function with zeros $x_1$ and $ x_2$, poles $x_3$ and $x_4$ and equal to $1$ in $[x_1 x_3| x_2 x_4]$ and
$[x_1 x_4| x_2 x_3].$
Finally, $q[x_1^2| x_2 x_3]$ will be a quadratic function double zero $x_1$, poles $x_2$ and $x_3$ and equal to $1$ in $[x_1 x_4| x_2 x_3]$ and  
$[x_1 x_5| x_2 x_3].$

The next lemma is the central part of the proof of the theorem \ref{third}. To simplify its formulation and proof let us make some conventions. Suppose that five points $x_1, x_2, x_3, x_4$ and $y$ are given. 
Quadratic functions, like $q[x_1 x_3| x_2 y],$ should be understood with respect to these five points. By $\sum$ we will mean summation over $24$ elements of the group $S_4$ of  all permutations of symbols $x_i$ in the summand. For instance, 
$\sum q[x_1 x_2, x_3,y]$ means $\sum_{\sigma \in S_4} q[x_{\sigma(1)} x_{\sigma(2)}, x_{\sigma(3)},y].$

\begin{lemma} \label{l45}
 The element 
 $$
\left ( -\frac{1}{2} \sum [x_1] [x_2]^2 + \frac{1}{3} \sum [x_1] [x_2][x_3]\right )\otimes \left( [x_1 x_2| x_3 x_4]+[x_1 x_3| x_2 x_4]+[x_1 x_4| x_2 x_3]\right)
 $$
 coincides with
  $$
\left ( -\frac{1}{2} \sum [x_1] [x_2]^2 + \frac{1}{3} \sum [x_1] [x_2][x_3]\right )\otimes \left( [x_1] +[x_2]+ [x_3]+[x_4]-[y]\right)
 $$
 modulo the image of $J_4$.

  \end{lemma}
\begin{proof}
Denote these elements by $X$ and $Y$. They lie in $\dfrac{Sym^3 A \otimes A}{Sym^4 A}.$ Throughout the proof we will work in $\overline{gr_3},$ so will suppose that $[a]^3\otimes [b], [a]^2[b]\otimes [a], [a]^2[b]\otimes [b], $ and $[a]^2[b]\otimes [c]- [a][b]^2\otimes [c]$ vanish.

First, notice that the difference of $X$ and $Y$ is homogeneous, that is is not changed if some $[\lambda]$ in $A$ is added to all variables in $X-Y$. To see it notice that
$$ [x_1 x_2| x_3 x_4]+[x_1 x_3| x_2 x_4]+[x_1 x_4| x_2 x_3] -[x_1] -[x_2]- [x_3]-[x_4]+[y]$$
is of degree zero and 
$$-\frac{1}{2} \sum [x_1] [x_2]^2 + \frac{1}{3} \sum [x_1] [x_2][x_3]=\frac{1}{36}\sum (2[x_1]-[x_2]-[x_3])^3.$$
From the homogeneity of $X-Y$ and the fact that all elements in the image of $J_4$ are homogenious as well it follows that it is enough to prove the statement of the lemma, supposing that $y=\infty$ and neglecting it.  

I claim that $6(X-Y)$ coincides with the image under $J_4$ of the following element $G$:
$$\sum q[x_1 x_2| x_3 \infty]+\frac{1}{2}\sum q[x_1^2| x_2 x_3]+\frac{1}{4}\sum q[\infty^2| x_1 x_2]-\frac{1}{2}\sum q[x_1^2|x_2 \infty].$$

To check it we need to compute $J_4$ of each term. I claim that

$\frac{1}{6}J_4q[x_1 x_2| x_3 \infty]=\frac{1}{6}([x_1]+[x_2]-[x_3])^3 \otimes ([x_1 x_2|x_3 x_4]+[x_4]-[x_3])=([x_1][x_2]^2-[x_1][x_2][x_3]) \otimes ([x_1 x_2|x_3 \infty]+[x_4]-[x_3]).$

$\frac{1}{6}J_4q[x_1^2| x_2 x_3 ]=-([x_1][x_2]^2+[x_2][x_3]^2+[x_3][x_1]^2-2[x_1][x_2][x_3]) \otimes ([x_2 x_3|x_1 x_4]+[x_2 x_3 |x_1 \infty]-[x_3]-[x_2]).$ 

$\frac{1}{6}J_4q[\infty^2| x_1 x_2]=-[x_1][x_2]^2\otimes([x_1x_2|x_3 \infty]+[x_1x_2|x_4 \infty])$

$\frac{1}{6}J_4q[x_1^2| x_2 \infty]=-[x_1][x_2]^2\otimes([x_1x_3|x_2 \infty]+[x_1x_4|x_2 \infty])$

For instance,  $\frac{1}{6}J_4q[x_1^2| x_2 \infty]=\frac{1}{6}(2[x_1]-[x_2]-[\infty])^3 \otimes ([x_1x_3|x_2 \infty]+[x_1x_4|x_2 \infty]-[x_2]-[\infty])=\frac{1}{6}(2[x_1]-[x_2])^3 \otimes ([x_1x_3|x_2 \infty]+[x_1x_4|x_2 \infty]-[x_2])=\frac{1}{6}(8[x_1]^3-12[x_1]^2[x_2]+6[x_1][x_2]^2-[x_2]^3)\otimes ([x_1x_3|x_2 \infty]+[x_1x_4|x_2 \infty]-[x_2])=-[x_1][x_2]^2\otimes([x_1x_3|x_2 \infty]+[x_1x_4|x_2 \infty]).$

The first thing to check is that in $6 J_4G$ all terms of the type $*\otimes [x_1,x_2|x_3,\infty]$ vanish.
For this let's combine all the terms of this type in $J_4G$. Contribution from $\sum q[x_1 x_2| x_3 \infty]$ will be $2([x_1][x_2]^2-[x_1][x_2][x_3])\otimes [x_1,x_2|x_3,\infty]$. Term $\frac{1}{2}\sum q[x_1^2| x_2 x_3]$ will give $-([x_1][x_2]^2+[x_2][x_3]^2+[x_3][x_1]^2-2[x_1][x_2][x_3]) \otimes [x_1 x_2|x_3 \infty]$. The contribution from $\frac{1}{4}\sum q[\infty^2| x_1 x_2]$ will be $-[x_1][x_2]^2\otimes [x_1 x_2|x_3 \infty]$ and $-\frac{1}{2}\sum q[x_1^2|x_2 \infty]$ will give $([x_2][x_3]^2+[x_1][x_3]^2)\otimes [x_1 x_2|x_3 \infty]$. Since 
$2([x_1][x_2]^2-[x_1][x_2][x_3])-([x_1][x_2]^2+[x_2][x_3]^2+[x_3][x_1]^2-2[x_1][x_2][x_3]) -[x_1][x_2]^2+[x_2][x_3]^2+[x_1][x_3]^2$ vanishes, the statement follows.

The remaining check of the fact that $(X-Y)$ coincides  $\frac{1}{6}J_4 G$ is straightforward.
\end{proof}

Now we are ready to prove the "five-term" identity.

\begin{lemma}\label{l46}
The image under $L_4$ of $$[a]\wedge[b] \otimes\left(\sum_{i=0}^4 (-1)^i \{r(x_0,..., \hat{x_i}, ..., x_4)\}_2\right )$$ lies in the image of $J_4.$ 
\end{lemma}
\begin{proof}
Without loss of generality we may suppose that $x_1=a, x_2=b, x_3=c, x_4=d$ and  $x_5=\infty.$ By definition of $L$ the image of the element $a \wedge b \otimes (a,x,b,\infty)$ equals $\frac{1}{2}[a]^2[b]\otimes [x]$ and of the element $a \wedge b \otimes (b,x,a,\infty)$ equals $-\frac{1}{2}[a]^2[b]\otimes [x].$ By lemma \ref{l43}, the following equalities of cross-ratios hold:

$r(b,c,d,\infty)=r(\infty,a,[ca|b\infty],b)=r(b,[ca|b\infty],a,\infty), $

$r(a,c,d,\infty)=r(\infty,b ,[cb|a\infty],a)=r(a,[cb|a\infty],b,\infty), $

$r(a,b,d,\infty)=\frac{r(a,d ,b,\infty)}{1-r(a,d ,b,\infty)}, $

$r(a,b,c,\infty)=\frac{r(a,c ,b,\infty)}{1-r(a,c ,b,\infty)}, $

$r(a,b,c,d)=1-r(a,c ,b,\infty)=1-r(b,[ab|d\infty],a,\infty).$

So, $\sum_{i=0}^4 (-1)^i \{r(x_0,..., \hat{x_i}, ..., x_4)\}_2=\{r(b,c,d,\infty)\}_2-\{r(a,c,d,\infty)\}_2+\{r(a,b,d,\infty)\}_2-\{r(a,b,c,\infty)\}_2+\{r(a,b,c,d)\}_2.$ By lemma \ref{l43}, 

$
L_4(a\wedge b \otimes \{r(b,c,d,\infty)\}_2)=-L_4(a\wedge b \otimes \{r(a,[ca|b\infty],b,\infty)\}_2)=-\frac{1}{2}[a]^2[b]\otimes [ca|b\infty],
$

$
L_4(a\wedge b \otimes \{r(a,c,d,\infty)\}_2)=L_4(a\wedge b \otimes \{r(a,[cb|a\infty],b,\infty)\}_2)=\frac{1}{2}[a]^2[b]\otimes [cb|a\infty],
$

$
L_4(a\wedge b \otimes \{r(a,b,d,\infty)\}_2)=-L_4(a\wedge b \otimes \{r(a,d,b,\infty)\}_2)=-\frac{1}{2}[a]^2[b]\otimes [d],
$

$
L_4(a\wedge b \otimes \{r(a,b,c,\infty)\}_2)=-L_4(a\wedge b \otimes \{r(a,c,b,\infty)\}_2)=-\frac{1}{2}[a]^2[b]\otimes [c],
$

$
L_4(a\wedge b \otimes \{r(a,b,c,d)\}_2)=L_4(a\wedge b \otimes \{r(b,a,\infty,[ab,c\infty] )\}_2)=-\frac{1}{2}[a]^2[b]\otimes [ab|d\infty].
$

So, we need to show that an element 
$$P=[a]^2[b]\otimes ([ca|b\infty]+[cb|a\infty]+[ab|c\infty]+[d]-[c])$$
lies in the image on $J_4.$ For this let us apply lemma \ref{l45} to five elements 
$x_1=a,x_2=b,x_3=c, x_4=\infty, y =d.$ we will have the following relation in $\overline{gr_3}:$
$$
(-[a]^2[b]-[b]^2[c]-[c]^2[a]+[a][b][c]) \otimes ([ca|b\infty]+[cb|a\infty]+[ab|c\infty] )=
$$
$$
(-[a]^2[b]-[b]^2[c]-[c]^2[a]+[a][b][c]) \otimes ([a]+[b]+[c]-[d]).
$$
It is easy to see that image under $J_4$ of the element $[ba||c \infty]+[ab|c\infty]$
equals  
$([a]^2[b]-[a][b][c]) \otimes ([ca|b\infty]+[cb|a\infty]+[ab|c\infty] +[d]-2[c]),$ so in $\overline{gr_3}$ we have an equality
$$
(-[a]^2[b]+[a][b][c]) \otimes ([ca|b\infty]+[cb|a\infty]+[ab|c\infty])=
(-[a]^2[b]+[a][b][c]) \otimes (2[c]-[d]).
$$ 
Similarly we have
$$
([a]^2[c]-[a][b][c]) \otimes ([ca|b\infty]+[cb|a\infty]+[ab|c\infty])=
([a]^2[c]-[a][b][c]) \otimes (2[b]-[d]),
$$
$$
([b]^2[c]-[a][b][c]) \otimes ([ca|b\infty]+[cb|a\infty]+[ab|c\infty])=
([b]^2[c]-[a][b][c]) \otimes (2[a]-[d]).
$$
Summation of these four elements shows that $P$ lies in the image of $J_4.$
\end{proof}

From lemmas \ref{l43} and \ref{l46} we conclude that $L_4$  is correctly defined map from $(A \wedge A ) \otimes B_2(\mathbb{C})$, to $\dfrac{\mathbb{S}^{31}A}{ImJ_4},$ such that $L_4 \circ K_4$ vanishes on $\dfrac{\mathbb{S}^{31}A}{ImJ_4}.$ From this it follows that $Ker(K_4)=Im(J_4)$ and the theorem \ref{third} is proved.

\end{document}